\documentclass{amsart}

\usepackage{latexsym,amssymb,amsthm,amsmath,amscd,multirow,graphics,multicol}
\usepackage{amssymb}

\theoremstyle{plain}  
\newtheorem{theorem}{Theorem}[section]

\newtheorem{lemma}{Lemma}[section]

\newtheorem{proposition}{Proposition}[section]

\newtheorem{corollary}{Corollary}[section]

\numberwithin{equation}{section}

\theoremstyle{remark}
\newtheorem{remark}{Remark}[section]

 \numberwithin{equation}{section}
\def\<{\left < }
\def\>{\right >}
\def\({\left ( }
\def\){\right )}
\def\e{\eqref}
\def\o{\omega}
\begin{document}

\vskip.3 in

\title[A Wintgen type inequality  and its applications]
{A Wintgen type inequality for surfaces in 4D neutral pseudo-Riemannian space forms and its applications to minimal immersions}

\author[B.-Y. Chen]{Bang-Yen Chen}

 \address{Department of Mathematics, 
	Michigan State University, East Lansing, Michigan 48824--1027, USA}
	
	\email{bychen@math.msu.edu}

\begin{abstract}  Let $M$ be a space-like surface immersed in a 4-dimensional pseudo-Riemannian space form $R^4_2(c)$ with constant sectional curvature $c$ and index two. In the first part of this article,  we prove that the Gauss curvature $K$, the normal curvature  $K^D$, and mean curvature vector $H$ of $M$ satisfy the general inequality: $K+K^D\geq \<H,H\>+c$. In the second part,  we investigate space-like minimal surfaces in $R^4_2(c)$ which satisfy the equality case of the inequality identically. Several classification results in this respect are then obtained.
\end{abstract}


 
\maketitle

\section{Introduction.}

Let $\mathbb E_{t}^m$ denote the  pseudo-Euclidean $m$-space equipped with  pseudo-Euclidean metric of  index $t$ 
given by
\begin{align}\label{1.1}g_{t}= -\sum_{i=1}^{t} dx_{i}^{2} + \sum_{j=t+1}^{n}dx_{j}^{2}, \end{align}
where $(x_{1},\ldots,x_{m})$ is a rectangular coordinate system of  $\mathbb E_{t}^m$. 

 We put
\begin{align} &\label{1.2} S^k_s(c)=\left\{x\in \mathbb E^{k+1}_s: \<x,x\>=\frac{1}{c}>0\right\},\\&\label{1.3} H^k_s(c)=\left\{x\in \mathbb E^{k+1}_{s+1}: \<x,x\>=\frac{1}{c}<0\right\},\end{align}
where $\<\;\,,\:\>$ is the associated inner product. 
Then $S^k_s(c)$ and $H^k_s(c)$ are complete pseudo-Riemannian manifolds of constant curvature $c$ and with index $s$, which are known as {\it pseudo-Riemannian $k$-sphere} and the {\it pseudo-hyperbolic $k$-space}, respectively. The pseudo-Riemannian manifolds $\mathbb E^k_s, S^k_s(c)$ and $H^k_s(-c)$ are called {\it pseudo-Riemannian space forms}.

A  vector $v$ is called {\it space-like} (respectively, {\it time-like}) if  $\<v,v\>>0$ (respectively, $\<v,v\><0$). A vector $v$ is called  {\it light-like} if it is  nonzero and it satisfies $\<v,v\>=0$. A surface $M$ in a pseudo-Riemannian manifold is called {\it space-like} if each nonzero tangent vector is space-like.

Let $M$ be a space-like surface immersed in a 4-dimensional pseudo-Riemannian space form  $R^4_2(c)$ with constant sectional curvature $c$ and index 2. In section 3, we recall a minimal immersion of $H^2(-\frac{1}{3})$ into the neutral pseudo-hyperbolic 4-space $H^4_2(-1)$ discovered recently by the author in \cite{c4}. In section 4,
we prove  that the Gauss curvature $K$, the normal curvature $K^D$, and mean curvature vector $H$ of $M$ in $R^4_2(c)$ satisfy the following general inequality:
 \begin{align} \label{1.4}K+K^D\geq \<H,H\>+c.\end{align} 
 In this section, we also show that there exist many minimal space-like surfaces which satisfy the equality case of this inequality.
In section 5, we investigate space-like minimal surfaces in the neutral pseudo-hyperbolic 4-space $H^4_2(-1)$ which satisfy the equality case of the inequality. In particular, we prove that if $K+1$ is a logarithm-harmonic function, then the minimal surface satisfies the equality case of \e{1.4} identically if and only if, up to rigid motions of $H^4_2(-1)$, the minimal surface is congruent to the recently discovered minimal surface described in section 3.
In the last two sections, we study minimal space-like surfaces in $\mathbb E^4_2$ and in $S^4_2$ which satisfy the equality case of \e{1.4}. Several classification results in this respect are then obtained.

\section{Preliminaries.}

\subsection{Basic formulas and definitions}
 Let  $R^4_2(c)$ denote the 4-dimensional  neutral pseudo-Riemannian space form  of constant  curvature $c$ and with index two.  Then the curvature tensor $\tilde R$ of $R^4_2(c)$ is given by
\begin{align} &\label{2.1}\tilde R(X,Y)Z= c\{ \left<X,Z\right>  Y-\left<Y,Z\right>X\}  \end{align} for vectors $X,Y,Z$ tangent to $R^4_2(c)$.
   Let $\psi:M \to R^4_2(c)$ be an isometric immersion of a space-like surface $M$ into $R^4_2(c)$.      
Denote by $\nabla$ and $\tilde\nabla$ the Levi-Civita connections on $M$ and $R^4_2(c)$, respectively. 

For vector fields $X,Y$ tangent to $M$ and vector field $\xi$ normal to $M$, the formulas of Gauss and Weingarten are given respectively by (cf. \cite{c1,c2,O}):
\begin{align} &\label{2.2}\tilde \nabla_XY=\nabla_XY+h(X,Y), \;\;
\\& \label{2.3}\tilde \nabla_X \xi=-A_\xi X+D_X\xi,\end{align} where $\nabla_X Y$ and $A_\xi X$ are the tangential components and $h(X,Y)$ and $D_X\xi$ are the normal components of $\tilde \nabla_XY$ and $\tilde \nabla_X \xi$, respectively. These formulas define the second
fundamental form $h$, the shape operator $A$, and the normal connection $D$ of $M$ in $R^4_2(c)$. 
 
 For each normal vector $\xi\in T_x^{\perp}M$,  $A_{\xi}$ is a symmetric endomorphism of the tangent space $T_xM,\,x\in M$. The shape operator and the second fundamental form are related by
\begin{align}\label{2.4} \<h(X,Y),\xi\>=\<A_{\xi}X,Y\>.\end{align}
 
 The {\it mean curvature vector}  $H$ of $M$ in $R^4_2(c)$ is defined by \begin{align} H=\(\frac{1}{2}\){\rm trace}\, h.\end{align}
The {\it mean curvature} of $M$ in $R^4_2(c)$ is defined to be $\sqrt{-\<H,H\>}$.

The equations of Gauss, Codazzi and Ricci are given
respectively by
\begin{align} &\label{2.6} R(X,Y)Z =  \left<X,Z\right>  Y-\left<Y,Z\right>X  + A_{h(Y,Z)}X- A_{h(X,Z)}Y,\\
&  \label{2.7} (\bar \nabla_X h)(Y,Z)=(\bar\nabla_Y h)(X,Z),\\ \label{2.8}
& \<R^{D}(X,Y)\xi,\eta\> = \<[A_{\xi},A_{\eta}]X,Y\>,\end{align}
for vector fields $X,Y,Z$  tangent to $M$ and $\xi$  normal to $M$, where $\bar\nabla h$ is defined by
\begin{equation}\begin{aligned}& (\bar\nabla_X h)(Y,Z) = D_X h(Y,Z) - h(\nabla_X Y,Z) - h(Y,\nabla_X Z),\end{aligned}\end{equation} and $R^D$ is the curvature tensor associated with the normal connection $D$, i.e., 
\begin{align}\label{2.10} &R^D(X,Y)\xi=D_XD_Y\xi-D_Y D_X\xi-D_{[X,Y]}\xi.\end{align}

For a space-like surface $M$ in  $R^4_2(c)$, the normal curvature $K^D$ is given by
\begin{align} \label{2.11} K^D=\<R^D(e_1,e_2)e_3,e_4\>. \end{align} 

A surface $M$ in $R^4_2(c)$ is called {\it a parallel surface\/} if   $\bar\nabla h=0$ holds identically.
An immersion $\psi$ of a surface $M$ in a pseudo-hyperbolic $4$-space $R^4_2(c)$ is called {\it full\/} if $\psi(M)$ does not lies in any totally geodesic submanifold of $R^4_2(c)$.  

The surface $M$ in $R^4_2(c)$ is called {\it totally umbilical} if the second fundamental form $h$ of $M$ satisfies $h(X,Y)=g(X,Y)\xi, \forall X,Y\in TM,$ for some normal vector field $\xi$. 

For an immersion $\psi:M\to H^4_2(-1)$  of  $M$ into $H^4_2(-1)$,
let $$\phi=\iota\circ \psi:M\to \mathbb E^5_3$$ denote the composition of $\psi$ with the standard inclusion $\iota:H^4_2(-1)\to \mathbb E^5_3$ via \e{1.3}.  

 Denote by $\tilde \nabla$ and
$\nabla$ the Levi-Civita connections of $\mathbb E^5_3$ and of $M$, respectively. Let $h$ be the second
fundamental form of $M$ in $H^4_2(-1)$. Since
 $H^{4}_2(-1)$ is totally umbilical  with one as its
mean curvature in  $\mathbb E^5_3$, we have  \begin{align}\label{2.16}\tilde \nabla_XY=\nabla_XY
+h(X,Y)+\phi\end{align} for $X,Y$ tangent to $M$.

\subsection{Connection forms}
Let $\{e_1,e_2\}$ be an orthonormal frame of the tangent bundle $TM$ of $M$. Then we have 
  \begin{align}\label{2.12} & \<e_1,e_1\>=\<e_2,e_2\>=1,\, \<e_1,e_2\>=0.\end{align}
  We  may choose  an orthonormal normal frame $\{e_3,e_4\}$ of $M$ in $R^4_2(c)$  such that 
  \begin{align} \label{2.13}& \<e_3,e_3\>=\<e_4,e_4\>=-1,\; \<e_3,e_4\>=0.\end{align}

For the orthonormal frame $\{e_1,e_2,e_3,e_4\}$,  we put
\begin{align}\label{2.14} \nabla_X e_1=\omega_1^2(X) e_2,\;\; D_X e_3=\omega_3^4(X)e_4,\end{align}
where $\o_1^2$ and $\o_3^4$ are the connection forms of the tangent and the normal bundles.

The  Gauss curvature $K$ and the normal curvature $K^D$ of $M$ are related with the connection forms $\omega_1^2$ and $\omega_3^4$ by
\begin{align}\label{2.15} d\omega_1^2=-K (*1),\;\; d\omega_3^4=-K^D (*1),\end{align} where $*$ is the Hodge star operator of $M$.

\subsection{Ellipse of curvature}
The {\it ellipse of curvature}  of a surface $M$ in  $R^4_2(c)$  is
the subset of the normal plane defined as $$\{h(v, v) \in T^\perp_pM:   |v |=1, v\in T_pM,\,p\in M\}.$$
To see that it is an ellipse, we consider an arbitrary orthogonal tangent frame
$\{e_1,e_2\}$. Put $h_{ij}=h(e_i,e_j),i,j=1,2$ and look at the following formula
\begin{align} \label{2.17}h(v,v)=H+\frac{h_{11}-h_{22}}{2}\cos 2\theta +h_{12}\sin 2\theta,\;\; v=\cos \theta e_1+\sin\theta e_2.\end{align}
 As $v$ goes once around the unit
tangent circle, $h(v,v)$ goes twice around the ellipse. The ellipse of curvature could degenerate into a line segment or a point. 

 The center of the ellipse is $H$. The ellipse of curvature is a circle if and only if
the following two conditions hold:\begin{align} \label{2.18} & |h_{11}-h_{22}|^2=4 |h_{12}|^2, \;\;\; \<h_{11}-h_{22},h_{12}\>=0.\end{align}

\section{A minimal immersion of $H^2(-\frac{1}{3})$ into $H^4_2(-1)$.}

In this section, we recall a minimal immersion of $H^2(-\frac{1}{3})$ into $H^4_2(-1)$ discovered recently in \cite{c4}.

  Consider the map  $\phi :{\bf R}^2\to \mathbb E^5_3$ defined by
\begin{equation}\begin{aligned}& \phi(s,t)=\text{\small$\Bigg($}\sinh \Big(\text{\small$\frac{2s}{\sqrt{3}}$}\Big)-\text{\small$ \frac{t^2}{3}$}-\(\text{\small$\frac{7}{8}+\frac{t^4}{18}$}\)e^{\frac{2s}{\sqrt{3}}}, t+\( \text{\small $\frac{t^3}{3}-\frac{t}{4}$}\)e^{\frac{2s}{\sqrt{3}}},\\& \hskip.1in \label{3.1}
 \text{\small $\frac{1}{2}$} + \text{\small $\frac{t^2}{2}$} e^{\frac{2s}{\sqrt{3}}},  t+ \(\text{\small $\frac{t^3}{3}$}+ \text{\small $\frac{t}{4}$}\)e^{\frac{2s}{\sqrt{3}}},
 \sinh \Big(\text{\small$\frac{2s}{\sqrt{3}}$}\Big)-\frac{t^2}{3}-\text{\small$\(\frac{1}{8}+\frac{t^4}{18}\)$}e^{\frac{2s}{\sqrt{3}}}\text{\small$\Bigg)$}. \end{aligned}\end{equation}
 The position vector $x$ of $\phi$ satisfies $\<x,x\>=-1$ and the induced metric via $\phi$ is  $g=ds^2+e^{\frac{2s}{\sqrt{3}}}dt^2.$ Thus, $\phi$ defines an isometric immersion  
 $\psi_{\phi}:H^2(-\frac{1}{3})\to H^4_2(-1)$  
  of the hyperbolic plane $H^2(-\frac{1}{3})$ of constant curvature $-\frac{1}{3}$ into $H^4_2(-1)$. This surface  satisfies $K^D=2K=-\frac{2}{3}$. So, we have $K+K^D=-1$.
  
  It was proved in \cite{c4} that, up to rigid motions,  $\psi_{\phi}:H^2(-\frac{1}{3})\to H^4_2(-1)$ is the only parallel minimal space-like surface lying fully in $H^4_2(-1)$.
  
  Recently, B.-Y. Chen and B. D. Suceav\u{a} proved in \cite{cs} the following classification  theorem.

\begin{theorem} \label{T:3.1} Let $\psi:M\to H^4_2(-1)$ be a minimal immersion of a space-like surface $M$ into $H^4_2(-1)$. If  the Gauss curvature $K$ and the normal curvature $K^D$ of $M$ are constant, then one of the following three statements holds.

\vskip.04in
{\rm (1)} $K=-1, K^D=0,$ and $\psi$ is totally geodesic.

\vskip.04in
{\rm (2)} $K=K^D=0$ and $\psi$ is congruent to an open part of the  minimal surface defined by
\begin{align}\label{3.2} &\hskip.0in L(u,v)=\text{\small$\frac{1}{\sqrt{2}}$}  \( \cosh u, \cosh v,0, \sinh u,  \sinh v \).\end{align}
 
\vskip.04in
{\rm (3)} $K^D=2K=-\frac{2}{3}$ and $\psi$ is congruent to an open part of the minimal surface $\psi_{\phi}:H^2(-\frac{1}{3})\to H^4_2(-1)$ induced from  \e{3.1}.
\end{theorem}

\begin{remark} If $M$ is a space-like totally geodesic surface in $H^4_2(-1)$, then the surface is congruent to an open part of the surface in $H^4_2(-1)$ induced from 
\begin{align}L(u,v)=(\cosh u\cosh v,0,0,\cosh u\sinh v, \sinh u)\end{align}
via \e{1.3}. \end{remark}

\section{A Wintgen type inequality for space-like surfaces in $R^4_2(c)$.} 

We need the following result for later use.

\begin{theorem}\label{T:4.1}  Let $M$ be a space-like  surface in a 4-dimensional neutral pseudo-Riemannian space form $R^4_2(c)$ of constant sectional curvature $c$. Then we have
 \begin{align} \label{4.1}& K+K^D\geq \<H,H\>+c \end{align}
 at every point in $M$.
 
 The equality sign of \e{4.1} holds at a point $p\in M$ if and only if, with respect to some suitable orthonormal frame $\{e_1,e_2,e_3,e_4\}$ at $p$,  the shape operators at $p$ take the forms:
   \begin{align}& \label{4.2} A_{e_3}=\begin{pmatrix} 2\gamma+\mu &0\\0 & \mu\end{pmatrix}, \; A_{e_4}=\begin{pmatrix} 0 &\gamma\\\gamma & 0 \end{pmatrix}.\end{align}\end{theorem}
\begin{proof} Assume that  $\psi:M\to R^4_2(c)$ is an isometric immersion of a  space-like  surface $M$ into a pseudo-Riemannian space form $R^4_2(c)$ of constant sectional curvature $c$. 

If $p\in M$ is totally geodesic point, i.e., $h(p)=0$, then we have $K(p)=-1$ and $K^D(p)=0$. So we have $K+K^D=c$ at $p$.

If $p\in M$ is a non-totally geodesic point, then we may choose an orthonormal  frame $\{e_1,e_2,e_3,e_4\}$ at $p$ such that the shape operators at $p$  satisfy
  \begin{align}& \label{4.3} A_{e_3}=\begin{pmatrix} \alpha &0\\0 & \mu\end{pmatrix}, \;\;\;  A_{e_4}=\begin{pmatrix} \delta &\gamma\\\gamma & -\delta \end{pmatrix}\end{align}
 for some functions $\alpha,\gamma,\delta,\mu$, with respect to $\{e_1,e_2,e_3,e_4\}$.  
 
 From \e{2.4}, \e{2.12}, \e{2.13} and \e{4.3} we know that the second fundamental form of $\psi$ satisfies
 \begin{align}\label{4.4}  &h(e_1,e_1)=-\alpha e_3-\delta e_4,\; h(e_1,e_2)=-\gamma e_4,\;  h(e_2,e_2)=-\mu e_3+\delta e_4.\end{align}
 
It follows from \e{4.4} and the equation of Gauss that the Gauss curvature $K$, the normal curvature $K^D$ and the mean curvature vector $H$ of $M$ at $p$ satisfy
 \begin{align} \label{4.5}&K(p)=-\alpha\mu+\gamma^2+\delta^2+c,
 \\& \label{4.6} K^D(p)=\gamma(\mu-\alpha),
 \\& \label{4.7} H(p)=\frac{\alpha+\mu}{2}e_3.\end{align}

From \e{4.5}-\e{4.7} we have
\begin{equation} \begin{aligned}\label{4.8}   K(p)+K^D(p)&=\<H(p),H(p)\>+\frac{1}{4}(2\gamma-\alpha+\mu)^2+\delta^2+c \\& \geq \<H(p),H(p)\>+c.
\end{aligned}\end{equation}
Consequently, we  obtain  inequality \e{4.1}. 

If the equality case of \e{4.1} holds at  $p\in M$, then \e{4.8} implies that we have $\delta=0$ and $\alpha=2\gamma+\mu$. Hence, we derive \e{4.2} from \e{4.3}. 

Conversely, if we have \e{4.2} at $p\in M$, then it is easy to verify that the equality sign of \e{4.1} holds at $p$.
\end{proof}

\begin{remark} Inequality \e{4.1} is a pseudo-hyperbolic version of an inequality of P. Wintgen obtained in \cite{W} (see, also \cite{GR}).
\end{remark}

\begin{remark} Every space-like  totally umbilical surface in $R^4_2(c)$ satisfies the equality case of \e{4.1} identically.
\end{remark}

\begin{remark} It follows from Theorem \ref{T:4.1} that if a space-like surface $M$ in $R^4_2(c)$ satisfies the equality case of inequality \e{4.1} identically, then $M$ is a Chen surface (in the sense of \cite{dpv,G}).
\end{remark}

\begin{remark} It follows from Theorem \ref{T:4.1} and conditions in \e{2.18} that if a space-like surface $M$ in $R^4_2(c)$ satisfies the equality case of inequality \e{4.1} identically, then it  has circular ellipse of curvature.
\end{remark}

\begin{remark} The minimal surface given by $\psi_\phi: H^2(-\frac{1}{3}) \to H^4_2(-1)$ discovered in \cite{c4} satisfies the equality case of \e{4.1} identically (with $H=0,c=-1)$.
\end{remark}

\begin{remark} On the neutral pseudo-Euclidean 4-space $\mathbb E^4_2$ equipped with the metric \begin{align}\label{4.9} g_2=-dx_1^2-dx_2^2+dx_3^2+dx_4^2,\end{align} we may consider the canonical complex coordinate system $\{z_1,z_2\}$ with $$z_1=x_1+{\bf i}\, x_2, z_2=x_3+{\bf i} \, x_4.$$ The complex structure on $\mathbb E^4_2$ obtained in this way is called the {\it standard complex structure on} $\,\mathbb E^4_2$.
In this way, we can regard ${\mathbb E}^4_2$ as a Lorentzian complex plane ${\bf C}^2_1$. \end{remark}

\begin{lemma} \label{L:4.1}  Every  space-like holomorphic curve in ${\bf C}^2_1$ satisfies the equality case of inequality \e{4.1} identically $($with $H=c=0).$ \end{lemma}
\begin{proof} Let $\psi:M\to {\bf C}^2_1$ be a holomorphic space-like curve in ${\bf C}^2_1$. Let $e_1$ be a unit tangent vector field of $M$. Then $e_2=Je_1$ is a unit tangent vector field of $M$ which is perpendicular to $e_1$. Consider an orthonormal normal frame $\{e_3,e_4\}$ of $M$ in ${\bf C}^2_1$ with $e_4=Je_3$. Then it follows from $\tilde \nabla_X J=0$ that \begin{align}\label{4.10}A_{e_4}X=JA_{e_3}X,\;\; \forall X\in TM.\end{align}
By applying \e{4.10} we know that the shape operator $A$ satisfies
  \begin{align}& \label{4.11} A_{e_3}=\begin{pmatrix} a &b\\b & -a\end{pmatrix}, \;\;\;  A_{e_4}=\begin{pmatrix} -b &a\\ a &b \end{pmatrix}\end{align}
 for some functions $a,b$, with respect to $\{e_1,e_2,e_3,e_4\}$.  
 
 By applying \e{4.11} we obtain $H=0$ and 
 $K=-K^D=2(a^2+b^2)$. Therefore, we obtain the equality case of \e{4.1} identically.
\end{proof}

\section{An application to minimal surfaces in $H^4_2(-1)$.}

Recall that a function $f$ on a space-like surface $M$ is called {\it   logarithm-harmonic}, if  $\Delta (\ln f)=0$ holds identically on $M$, where $\Delta (\ln f):=*d*(\ln f)$ is the Laplacian of $\ln f$ and $*$ is the Hodge star operator. A function $f$ on $M$ is called {\it subharmonic} if $\Delta f\geq 0$ holds everywhere on $M$.

In this section, we establish the following simple geometric characterization of the minimal immersion $\psi_{\phi}: H^2(-\frac{1}{3})\to H^4_2(-1)$ given in section 3.

\begin{theorem}\label{T:5.1} Let $\psi:M\to H^4_2(-1)$ be a non-totally geodesic, minimal immersion of a space-like surface $M$ into $H^4_2(-1)$. Then  \begin{align}\label{5.1} K+K^D\geq  -1\end{align} holds identically on $M$.

If $K+1$ is   logarithm-harmonic,  then the equality sign of \e{5.1} holds identically if and only if  $\psi:M\to H^4_2(-1)$ is congruent to an open portion of the immersion $\psi_{\phi}:H^2(-\frac{1}{3})\to H^4_2(-1)$ which is induced from the map $\phi:{\bf R}^2\to \mathbb E^5_3$ defined by
  \begin{equation}\begin{aligned}&\label{5.2} \phi (s,t)=\Bigg(\sinh \Big(\text{\small$\frac{2s}{\sqrt{3}}$}\Big)-\text{\small$ \frac{t^2}{3}$}-\(\text{\small$\frac{7}{8}+\frac{t^4}{18}$}\)e^{\frac{2s}{\sqrt{3}}}, t+\( \text{\small $\frac{t^3}{3}-\frac{t}{4}$}\)e^{\frac{2s}{\sqrt{3}}},\\& \hskip.1in 
 \text{\small $\frac{1}{2}$} + \text{\small $\frac{t^2}{2}$} e^{\frac{2s}{\sqrt{3}}},  t+ \(\text{\small $\frac{t^3}{3}$}+ \text{\small $\frac{t}{4}$}\)e^{\frac{2s}{\sqrt{3}}},
 \sinh \Big(\text{\small$\frac{2s}{\sqrt{3}}$}\Big)-\frac{t^2}{3}-\text{\small$\(\frac{1}{8}+\frac{t^4}{18}\)$}e^{\frac{2s}{\sqrt{3}}}\Bigg). \end{aligned}\end{equation}
\end{theorem}
\begin{proof} Assume that $\psi:M\to H^4_2(-1)$ is a non-totally geodesic, minimal immersion of a space-like surface $M$ into $H^4_2(-1)$. Then the mean curvature vector $H$ vanishes identically.  Thus, we obtain inequality \e{5.1} from \e{4.1}.

From now on, let us assume that  $M$ is a minimal space-like surface in $H^4_2(-1)$ which satisfies the equality case of \e{5.1} identically. Then Theorem \ref{P:4.1} implies that there exists 
an orthonormal frame $\{e_1,e_2,e_3,e_4\}$ such that  the shape operators  take the following special forms:
   \begin{align}& \label{5.3} A_{e_3}=\begin{pmatrix} \gamma &0\\0 & -\gamma\end{pmatrix}, \; A_{e_4}=\begin{pmatrix} 0 &\gamma\\\gamma & 0 \end{pmatrix}.\end{align}
   Hence, after applying \e{2.4}, \e{2.12} and \e{2.13}, we obtain
    \begin{align}\label{5.4}  &h(e_1,e_1)=-\gamma e_3,\; h(e_1,e_2)=-\gamma e_4,\;  h(e_2,e_2)=\gamma e_3.\end{align}

It follows from \e{2.14}, \e{5.4} and the equation of Codazzi that
\begin{align} &\label{5.5} e_1\gamma=-2\gamma \omega_1^2(e_2) +\gamma \omega_3^4(e_2),
\\& \label{5.6} e_2\gamma=2\gamma \omega_1^2(e_1)-\gamma \omega_3^4(e_1).\end{align}

Since the star operator satisfies $$*(d\gamma)=-(e_2\gamma )\omega^1+(e_1\gamma)\omega^2,$$ Eqs. \e{5.5} and \e{5.6} imply that
\begin{align} &\label{5.7} *d\gamma=\gamma(\omega_3^4-2 \omega_1^2).\end{align}
Thus, we find from \e{2.15} and \e{5.7} that
\begin{equation}\begin{aligned} \label{5.8} \Delta \gamma &=\gamma (2K-K^D)+\frac{*(d\gamma \wedge *d\gamma)}{\gamma},\end{aligned}\end{equation}
where  $\Delta \gamma$ is the Laplacian of $\gamma$ defined by $\Delta \gamma=*d*d\gamma$.

From \e{5.8} we deduce that
\begin{align} &\label{5.9} \Delta \gamma=\gamma (2K-K^D)+\frac{|d\gamma|^2}{\gamma}.
\end{align}

On the other hand, it follows from $$\Delta (\ln (K+1))=*d*d(\ln (K+1)),\;\;\; K=2\gamma^2-1$$ that
\begin{equation}\begin{aligned} \label{5.10} \Delta (\ln (K+1))&=\frac{(K+1) \Delta K-*(dK\wedge *dK)}{(K+1)^2}
\\&  = \frac{2\gamma^2 (4|d\gamma|^2 +4\gamma\Delta\gamma)-16\gamma^2 |d\gamma|^2}{(K+1)^2} 
\\&  = \frac{2\gamma\Delta\gamma-2 |d\gamma|^2}{\gamma^2}.
\end{aligned}\end{equation}
Therefore,  \e{5.9} and \e{5.10} yield
\begin{equation}\begin{aligned} \label{5.11} \Delta (\ln (K+1))&=2(2K-K^D).
\end{aligned}\end{equation}

Now, let us assume that $K+1$ is a  logarithm-harmonic function, then Eq. \e{5.11} gives $K^D=2K$. Hence, after combining this with the equality case of \e{5.1}, we obtain that $K^D=2K=-\frac{2}{3}$. Therefore, by applying Theorem \ref{T:3.1}, we conclude that, up to rigid motions of $H^4_2(-1)$, the minimal surface is an open portion of the minimal surface $\psi_{\phi}:H^2(-\frac{1}{3})\to H^4_2(-1)$ induced from  the map \e{5.2}.

The converse can be verified by direct computation.
\end{proof}

\begin{corollary}\label{C:5.1} Let $\psi:M\to H^4_2(-1)$ be a minimal immersion of a space-like surface $M$ of constant Gauss curvature into $H^4_2(-1)$. 
Then the equality sign of \e{5.1} holds identically if and only if  one of the following two statements holds.

\vskip.04in
{\rm (1)} $K=-1, K^D=0,$ and $\psi$ is totally geodesic.

\vskip.04in
{\rm (2)}  $K^D=2K=-\frac{2}{3}$ and $\psi$ is congruent to an open part of the minimal surface $\psi_{\phi}:H^2(-\frac{1}{3})\to H^4_2(-1)$ induced from  \e{3.1}.
\end{corollary}
\begin{proof} Let $\psi:M\to H^4_2(-1)$ be a minimal immersion of a space-like surface $M$  into $H^4_2(-1)$. If the Gauss curvature $K$ is constant and the equality sign of \e{5.1} holds, then both $K$ and $K^D$ are constant. Therefore, by applying Theorem \ref{T:3.1}, we obtain either Case (1) or Case (2).

The converse is trivial. \end{proof}

\section{Space-like minimal surfaces in $\mathbb E^4_2$  satisfying the equality.}

It follows from Lemma \ref{L:4.1} that there exist infinitely many non-totally geodesic, minimal space-like surfaces in $\mathbb E^4_2$ which satisfy the equality case of inequality \e{4.1} identically (with $H=c=0$). 

On the other hand, we have the following.

\begin{proposition}\label{P:61} Let $\psi:M\to {\mathbb E}^4_2$ be a minimal immersion of a space-like surface $M$  into the pseudo-Euclidean 4-space $\mathbb E^4_2$. Then  \begin{align}\label{6.1} K\geq -K^D\end{align} holds identically on $M$.

If $M$ has constant Gauss curvature, then 
 the equality sign of \e{6.1} holds identically if and only if $M$ is a totally geodesic surface.
\end{proposition}
\begin{proof} Let $\psi:M\to \mathbb E^4_2$ be a minimal immersion of a space-like surface $M$  into $\mathbb E^4_2$. Then inequality \e{6.1} follows immediately from Theorem \ref{P:4.1}.

Assume that  the equality case of \e{6.1} holds identically. Then Theorem \ref{P:4.1} implies that there exists 
an orthonormal frame $\{e_1,e_2,e_3,e_4\}$ such that  the shape operator $A$ takes the special forms:
   \begin{align}& \label{6.2} A_{e_3}=\begin{pmatrix} \gamma &0\\0 & -\gamma\end{pmatrix}, \; A_{e_4}=\begin{pmatrix} 0 &\gamma\\\gamma & 0 \end{pmatrix}.\end{align}
From  \e{6.2} and the equation of Codazzi we find
\begin{align} &\label{6.3} e_1\gamma=-2\gamma \omega_1^2(e_2) +\gamma \omega_3^4(e_2),
\\& \label{6.4} e_2\gamma=2\gamma \omega_1^2(e_1)-\gamma \omega_3^4(e_1).\end{align}

 If the Gauss curvature $K$ is a nonzero constant, then the function $\gamma$ is a nonzero constant. In this case, \e{6.3} and \e{6.4} imply that 
\begin{align} &\label{6.5}2 \omega_1^2= \omega_3^4.\end{align}
Thus, after taking exterior differentiation of \e{6.5} and applying \e{2.15}, we obtain $2 K=K^D$. Combining this with the equality of \e{6.1} yields $K=0$, which is a contradiction. Therefore,  we must have $K=2\gamma^2=0$. Consequently, $M$ is totally geodesic in $\mathbb E^4_2$.
 
 The converse is trivial.
\end{proof}

\begin{proposition}\label{P:6.2} Let $\psi:M\to {\mathbb E}^4_2$ be a minimal immersion of a space-like surface $M$  into $\mathbb E^4_2$. We have

\vskip.04in
{\rm (1)}  If the equality sign of \e{6.1} holds identically, then $K$ is a non-logarithm-harmonic function.

\vskip.04in
{\rm (2)} If $M$ contains no totally geodesic points and  the equality sign of \e{6.1} holds identically on $M$, then $\ln K$ is  subharmonic. \end{proposition}
\begin{proof} 
Assume that $M$ is a minimal space-like surface in $\mathbb E^4_2$ which satisfies the equality case of \e{6.1}, i.e., $K=-K^D$ identically. Then Theorem \ref{P:4.1} implies that there exists 
an orthonormal frame $\{e_1,e_2,e_3,e_4\}$ such that  the shape operator $A$  takes the special forms given by \e{6.2}.

 From  \e{6.2} and the equation of Codazzi we obtain \e{6.3} and \e{6.4}.
Thus, we may apply  the same arguments as in section 5 to obtain that
\begin{equation}\begin{aligned} \label{6.6} \Delta (\ln K)&=2(2K-K^D)
\end{aligned}\end{equation}
at each non-totally geodesic point.
Hence, after combining this with $K=-K^D$, we obtain $K=0$. But this is impossible, since in this case $\ln K$ is undefined.
This proves statement (1).

Next, assume that $M$ contains no totally geodesic points and that the equality sign of \e{6.1} holds identically on $M$. Then, we find from \e{6.2}, \e{6.6} and $K=-K^D$ that $$\Delta (\ln K)=6 K=12\gamma^2>0,$$ which implies that $\ln K$ is a subharmonic function. This proves statement (2).
\end{proof}

\section{Space-like minimal surfaces in $S^4_2(1)$  satisfying the equality.}

Now, we study space-like minimal surfaces in $S^4_2(1)$ satisfying the equality case of inequality \e{4.1}.

\begin{proposition}\label{P:7.1} Let $\psi:M\to S^4_2(1)$ be a minimal immersion of a space-like surface $M$ into the neutral pseudo-sphere $S^4_2(1)$. Then  \begin{align}\label{7.1} K+K^D\geq 1\end{align} holds identically on $M$.

If $M$ has constant Gauss curvature, then the equality sign of \e{7.1} holds identically if and only if $M$ is a totally geodesic surface.
\end{proposition}
\begin{proof} Assume that $\psi:M\to S^4_2(1)$ is a  minimal immersion of a space-like surface $M$ into $S^4_2(1)$. Then inequality \e{4.1} in Theorem \ref{P:4.1} reduces to inequality \e{7.1}.

Suppose that  the equality case of \e{7.1} holds identically on $M$, then Theorem \ref{P:4.1} implies that there exists 
an orthonormal frame $\{e_1,e_2,e_3,e_4\}$ such that  the shape operator $A$  takes the following special forms:
   \begin{align}& \label{7.2} A_{e_3}=\begin{pmatrix} \gamma &0\\0 & -\gamma\end{pmatrix}, \; A_{e_4}=\begin{pmatrix} 0 &\gamma\\\gamma & 0 \end{pmatrix}.\end{align}
   Hence, by applying \e{2.4}, \e{2.12} and \e{2.13}, we know that the second fundamental form $h$ satisfies
    \begin{align}\label{7.3}  &h(e_1,e_1)=-\gamma e_3,\; h(e_1,e_2)=-\gamma e_4,\;  h(e_2,e_2)=\gamma e_3.\end{align}

It follows from \e{7.3} that the Gauss curvature of $M$ is given by $K=1+2\gamma^2$.  Now, let us assume that the Gauss curvature $K$ is constant. Then $\gamma$ is constant. Let us suppose that $M$ is non-totally geodesic in $S^4_2(1)$. 
Then, by applying   \e{7.3} and the equation of Codazzi, we find
\begin{align} &\label{7.4}2 \omega_1^2= \omega_3^4.\end{align}
Thus, after taking exterior differentiation of \e{7.7} and applying \e{2.15}, we obtain \begin{align} &\label{7.5}2 K=K^D.\end{align}
By combining \e{7.5} with the equality of \e{7.1}, we get $K^D=\frac{2}{3}$.  
 
 On the other hand, it follows from \e{2.11} and  \e{7.2} that $K^D=-2\gamma^2\leq 0$, which contradicts to $K^D=\frac{2}{3}$. Consequently, $M$ must be totally geodesic in $S^4_2(1)$. 

Conversely, if $M$ is totally geodesic in $S^4_2(1)$, then we have $K=1$ and $K^D=0$. So, we get $K+K^D=1$, which is exactly the equality case of \e{7.1}.
\end{proof}

Finally, we prove the following.

\begin{proposition}\label{P.7.2} Let $\psi:M\to S^4_2(1)$ be a minimal immersion of a space-like surface $M$ into $S^4_2(1)$. We have

\vskip.04in
{\rm (1)} If the equality sign of \e{7.1} holds identically, then $K-1$ is non-logarithm-harmonic.

\vskip.04in
{\rm (2)} If $M$ contains no totally geodesic points and if the equality case of \e{7.1} holds, then $\ln (K-1)$ is subharmonic.
\end{proposition}
\begin{proof} Assume that $M$ is a minimal space-like surface of $S^4_2(1)$ which satisfies the equality case of \e{7.1} identically,. Then we have $K+K^D=1$.  Moreover, from Theorem  \ref{P:4.1} we know that there exists 
an orthonormal frame $\{e_1,e_2,e_3,e_4\}$ such that  the shape operator $A$  satisfies
   \begin{align}& \label{7.6} A_{e_3}=\begin{pmatrix} \gamma &0\\0 & -\gamma\end{pmatrix}, \; A_{e_4}=\begin{pmatrix} 0 &\gamma\\\gamma & 0 \end{pmatrix}.\end{align}
Hence, we may applying the same arguments as in section 5 to obtain that
\begin{equation}\begin{aligned} \label{7.7} \Delta (\ln (K-1))&=2(2K-K^D).
\end{aligned}\end{equation}

If $K-1$ is logarithm-harmonic, then Eq. \e{7.7} yields $K^D=2K$. Thus, after combining this with the equality $K+K^D=1$, we obtain \begin{align}\label{7.8} K^D=\frac{2}{3}.\end{align}

On the other hand, we find from \e{7.6} that $K^D=-2\gamma^2\leq 0$, which contradicts to \e{7.8}. Consequently, $K-1$ cannot be  a logarithm-harmonic function. This proves statement (1).

Next, assume that $M$ contains no totally geodesic points and if the equality case of \e{7.1} holds. Then, we find from \e{7.6} and \e{7.7} that
\begin{equation}\begin{aligned} \label{7.9} \Delta (\ln (K-1))&=4(3\gamma^2+1)>0.
\end{aligned}\end{equation}
Hence, $\ln(K-1)$ is a subharmonic function. This proves statement (1).
\end{proof}


\begin{thebibliography}{99}


\bibitem{c1} B. Y. Chen, {\it Geometry of Submanifolds}, Mercer Dekker, New York, 1973.

\bibitem{c2} B.-Y. Chen, {\it Total Mean Curvature and Submanifolds of Finite Type}, World Scientific, New Jersey, 1984.

\bibitem{c3} B.-Y. Chen,  Riemannian submanifolds,  {\it Handbook of Differential Geometry}, Vol. I, 187--418, North-Holland, Amsterdam, 2000 (eds. F. Dillen and L. Verstraelen).


\bibitem{c4} B.-Y. Chen,   A minimal immersion of hyperbolic plane in neutral pseudo-hyperbolic 4-space and its characterization, Arch. Math.  {\bf 94} (2010), 291--299.

\bibitem{cs} B.-Y. Chen and B. D. Suceav\u{a}, Space-like minimal surfaces of constant curvature in  pseudo-hyperbolic 4-space $H^4_2(-1)$, {\it Taiwan. J. Math.} {\bf 15} (2011), no. 2, 523--541.

\bibitem{dpv} S. Decu, M. Petrovi\'c--Torga\v sev and L. Verstraelen, On the intrinsic Deszcz symmetries and the extrinsic Chen character of Wintgen ideal submanifolds, {\it Tamkang J. Math.} {\bf  41} (2010), no. 2, 109--116.

\bibitem{G} L. Gheysens, P.  Verheyen and L. Verstraelen,  Characterization and examples of Chen submanifolds, {\it  J. Geom.} {\bf 20} (1983), 47--62.

\bibitem{GR} I. V. Guadalupe nd L. Rodriguez,  Normal curvature of surfaces in space forms, {\it Pacific J. Math.} {\bf 106} (1983), 95--103.


\bibitem{O} B. O'Neill,  {\it Semi-Riemannian Geometry with Applications to Relativity}, Academic Press, New York, 1982.

\bibitem{W} P. Wintgen, Sur l'in\'egalit\'e de Chen-Willmore, {\it C. R. Acad. Sci. Paris}, {\bf 288} (1979),
993--995.
\end{thebibliography}
  \end{document}